\documentclass[16pt,oneside,reqno]{amsart}
\usepackage{amsfonts,amssymb,amscd,amsmath,enumerate,verbatim,calc}
\usepackage{color,soul}
\usepackage{mathtools}
\usepackage{graphicx}
\everymath{\displaystyle}
\usepackage{xcolor}
\usepackage{fancyhdr}
\usepackage{tikz}
\usetikzlibrary{shapes.geometric, arrows}
\tikzstyle{arrow} = [thick,->,>=stealth]
\tikzstyle{process} = [rectangle, minimum width=4cm, minimum height=2cm, text centered, text width=2.5cm, draw=green, fill=green!10]
\newcommand{\C}{\mathbb{C} }

\theoremstyle{plain}

\newtheorem{theorem}{Theorem}[section]
\newtheorem{corollary}[theorem]{Corollary}

\newtheorem{proposition}[theorem]{Proposition}

\theoremstyle{definition}
\newtheorem{definition}[theorem]{Definition}

\newtheorem{remark}[theorem]{Remark}
\newtheorem{example}[theorem]{Example}

\begin{document}
\title{On the Rank of a bicomplex matrix}
\author{amita}
\address{Department of Mathematics, IGNTU, Amarkantak, Madhya Pradesh 484887, India}

\email{amitasharma234@gmail.com}
\author{Mamta Amol Wagh}
\address{Department of Mathematics, University of Delhi, New Delhi 110078, India}
\email{mamtanigam@ddu.du.ac.in}

\author{Suman Kumar}
\address{Department of Mathematics, IGNTU, Amarkantak, Madhya Pradesh 484887, India}
\email{suman@igntu.ac.in}
\author{Akhil Prakash}
\address{Department of Mathematics, Aligarh Muslim University, Aligarh, Uttar Pradesh 202002, India}
\email{gm5537@myamu.ac.in}
\subjclass[IMS]{Primary 15A04, 15A30; Secondary 30G35}
%\date{\today}

\begin{abstract}
The paper explores the concept of the rank of a bicomplex matrix, delving into four distinct types of ranks and investigating conditions under which these ranks are equivalent. It also defines and analyzes the concept of idempotent row space and idempotent column space of a bicomplex matrix. Some examples and counter examples have been presented to substantiate the study. 
\end{abstract}
\maketitle

\section{Introduction and preliminaries}
In recent years, the theory of bicomplex numbers has become a thriving area of mathematical research, undergoing significant advancements and branching into new directions. In the historical development of bicomplex numbers, Segre is credited with its initial introduction \cite{segre1892}. Although not exhaustive, resources such as Price \cite{price2018introduction} offer a comprehensive foundation for multicomplex number theory. Notably, recent research efforts, (see \cite{alpay2014basics, futagawa1928, futagawa1932, {gervais2011finite}, {luna2015bicomplex}, riley1953, rochon2004, srivastava2008}) have further propelled advancements in this field.

In this section, we provide a summary of several properties associated with bicomplex numbers.

\noindent {\bf Bicomplex numbers:} A bicomplex number can be represented as
 ${u}_{1}  + {i}_{1} {u}_{2}  + {i}_{2} {u}_{3}  + {i}_{1} {i}_{2} {u}_{4}$, where ${u_k}$ are the real numbers for $k=1$ to $4$   and $i_1$, $i_2$ are  unit vectors satisfying the properties $i_1 i_2=i_2 i_1, i_1^2=i_2^2=-1$. For the sake of convenience, we employed certain symbols $\C_0, \C_1$, and $\C_2$, for the set of real numbers, the set of complex numbers, and the set of bicomplex numbers respectively. The set of bicomplex numbers can be represented in three different ways
\begin{eqnarray*}
 \mathbb{C}_{2} &=& \{{u}_{1}  + {i}_{1} {u}_{2}  + {i}_{2} {u}_{3}  + {i}_{1} {i}_{2} {u}_{4}\;:\; {u}_{1} , {u}_{2} , {u}_{3} , {u}_{4} \in\mathbb{C}_{0}\}, \\
 \mathbb{C}_{2} &=& \{{z}_{1} + {i}_{2} {z}_{2}\;:\;  {z}_{1}, {z}_{2} \in \mathbb{C}_{1}\},\\
\mathbb{C}_{2} &=& \{1_\xi e_1 + 2_\xi e_2\;:\;  {1_\xi=({z}_{1}-i_1 {z}_{2}), 2_\xi=({z}_{1}+i_1 {z}_{2}) \in \mathbb{C}_{1}}\}.
\end{eqnarray*}

The introduction of two numbers, $e_1 = \frac{(1 + i_1 i_2)}{2}$ and $e_2 = \frac{(1 - i_1 i_2)}{2}$, significantly simplifies the system of bicomplex numbers $\C_2$. These numbers, $e_1$ and $e_2$ serve as zero divisors in $\C_2$. Furthermore, the numbers $e_1$ and $e_2$ form a basis for $\C_2$, which is called the idempotent basis. Thus, every bicomplex number $\xi$ can be written in a unique way as:
$\xi = 1_\xi e_1 + 2_\xi e_2$. This representation is called the idempotent representation of a bicomplex number $\xi$.
The idempotent representation gives algebraic structure to the bicomplex space, for instance the product of two bicomplex numbers can be seen component-wise in the above idempotent representation, i.e. if $\xi=1_\xi e_1 + 2_\xi e_2$ and $\eta=1_\eta e_1 + 2_\eta e_2$ in $\C_2$, then    
\begin{eqnarray*}
\xi \cdot \eta = (1_\xi \cdot 1_\eta) e_1 + (2_\xi \cdot 2_\eta) e_2.  \end{eqnarray*}
The set of singular elements in $\C_2$ is denoted by $O_2$ and defined as the collection of all non-invertible elements in $\C_2$.

\begin{definition}
\noindent({\bf Principal ideals} \cite{price2018introduction})
The principal ideals in $\C_2$ determined by $e_1$ and $e_2$, denoted by $I_1$ and $I_2$, respectively are defined as:
\begin{eqnarray*}
I_1=:\{\xi e_1:\xi \in \C_2\}=\{1_\xi e_1:1_\xi \in \C_1\},\\
I_2=:\{\xi e_2:\xi \in \C_2\}=\{2_\xi e_2:2_\xi \in \C_1\}.   
\end{eqnarray*}
\end{definition}
\begin{theorem}\label{th1}
\cite{price2018introduction} A bicomplex number $\xi=(Z_1+i_2Z_2)$ is singular if and only if $\xi \in I_1 \cup I_2$. 
\end{theorem}

\begin{theorem}\label{th2}
\cite{price2018introduction} A bicomplex number $\xi=(Z_1+i_2Z_2)$ is non-singular if and only if $\xi \notin I_1 \cup I_2$.
\end{theorem}

\begin{definition}\label{th5}
(\noindent{\bf Cartesian product}
\cite{anjali2023matrix}) The $n$-times Cartesian product of $\C_2$ is denoted by $\C_2^n$ and defined as
\begin{eqnarray*}
\mathbb C_2^n=\{(\xi_1,\xi_2,\ldots,\xi_n):\xi_i \in \C_2;i=1,2,\ldots,n\}.     
\end{eqnarray*}
Moreover, analogous to the idempotent representation of an element in $\C_2$, every element of $\C_2^n$ can be expressed uniquely as
\begin{eqnarray*}
 (\xi_1,\xi_2,\ldots,\xi_n)= (1_{\xi_1},1_{\xi_2},\ldots,1_{\xi_n})e_1 + (2_{\xi_1},2_{\xi_2},\ldots,2_{\xi_n}) e_2,
\end{eqnarray*}
where $(1_{\xi_1},1_{\xi_2},\ldots,1_{\xi_n}),(2_{\xi_1},2_{\xi_2},\ldots,2_{\xi_n})$ are $n$-tuples of complex numbers.
\end{definition}
\begin{remark}\label{th49}\cite{anjali2023matrix} 
If $(\xi_1,\xi_2,\ldots,\xi_n), (\eta_1,\eta_2,\ldots,\eta_n)\in \C_2^n$ and $\kappa \in \C_2$, then
\begin{itemize}
\item 
$(\xi_1,\xi_2,\ldots,\xi_n)\cdot (\eta_1,\eta_2,\ldots,\eta_n)=(\xi_1\eta_1,\xi_2\eta_2,\ldots,\xi_n\eta_n).$
\item 
$(\xi_1,\xi_2,\ldots,\xi_n)=(\eta_1,\eta_2,\ldots,\eta_n) \iff 1\xi_i=1\eta_i \ and \ 2\xi_i=2\eta_i\ \forall \ i=1,2,\ldots n$.
\item 
$\kappa(\xi_1,\xi_2,\ldots,\xi_n) = (\xi_1,\xi_2,\ldots,\xi_n) \kappa =(\kappa\xi_1 ,\kappa \xi_2,\ldots,\kappa \xi_n)$.
\end{itemize} 
Consequently, the product
\begin{eqnarray*}
(\xi_1,\xi_2,\ldots,\xi_n)e_i=(\xi_1e_i,\xi_2e_i,\ldots,\xi_n e_i)=(i_{\xi_1}e_i,i_{\xi_2}e_i,\ldots,i_{\xi_n} e_i) ; i=1,2.
\end{eqnarray*}
This conclusion arises from the fact that $\C_2^n$ is not only a $\C_2$ module but also exhibits the structure of a $\C_1$-algebra.
\end{remark}

\begin{definition}\label{th6}
\noindent{\bf (Bicomplex matrix)}
The set $\C_2^{n\times m}$ of bicomplex matrices of order ${n\times m}$ is defined as
\begin{eqnarray*}
\C_2^{n\times m}=:\{[\xi_{ij}]_{n\times m}:\xi_{ij}\in \C_2;1\leq i \leq n, 1\leq j \leq m \}.
\end{eqnarray*}
The set $\C_2^{n\times n}$ with respect to ordinary multiplication, addition and scalar multiplication forms an algebra over $\C_1$. Similar to the concept of idempotent representation for elements in $\C_2$, every bicomplex matrix can be uniquely expressed as 
\[ 
[\xi_{ij}]_{n\times m}=[1_{\xi_{ij}}]_{n\times m}e_1 + [2_{\xi_{ij}}]_{n\times m}e_2; 1\leq i \leq n, 1\leq j \leq m,  
\]
where $[1_{\xi_{ij}}]_{n\times m}$ and $[2_{\xi_{ij}}]_{n\times m}$ are complex matrices of order $n\times m$, i.e. if
$A=[\xi_{ij}]_{n\times m}\in \C_2^{n\times m}$ then $A=1_Ae_1+2_Ae_2$, where $1_A=[1_{\xi_{ij}}]_{n\times m}\in \C_1^{n\times m}$, $2_A=[2_{\xi_{ij}}]_{n\times m}\in \C_1^{n\times m}$ and $\C_1^{n\times m}$ denotes the set of all complex matrices of order ${n\times m}$.
\end{definition}

\begin{definition}
\noindent{\bf (Bicomplex singular and non-singular matrix)} If $A\in C_2^{n\times n}$, then $A$ is said to be non-singular if determinant of $A$ is non-singular element, i.e. $\det(A)\notin O_2$ and if determinant of $A$ is singular element, i.e. $\det(A)\in O_2$, then it is called singular.
    
\end{definition} 

\begin{theorem}\label{th7}\cite{price2018introduction}
If $A$ is a bicomplex matrix of order $n\times n$, then
$\det(A) = \det(1_A) e_1 + \det(2_A) e_2 $.
\end{theorem}

\begin{corollary}\label{th9}\cite{price2018introduction}
Let $A$ be a bicomplex matrix of order $n\times n$, then $A$ is non-singular if and only if $\det(1_A) \neq0$ and $\det(2_A) \neq0$.  \end{corollary}

\begin{corollary}\label{th10}\cite{price2018introduction}
Let $A$ be a bicomplex matrix of order $n\times n$, then $A$ is singular if and only if  either $\det(1_A)=0$ or $\det(2_A)=0$.
\end{corollary}

\section{rank of Bicomplex matrix} \label{th11}
This section deals with the rank, row rank, and column rank of a bicomplex matrix. Here we explore some results on these type of ranks. We define the idempotent row space and idempotent column space of a bicomplex matrix, also investigate some results. 

\begin{definition}
Let $V$ be a subset of $\C_1^{n}$, then we define the set $Ve_i$ as follows
\begin{eqnarray*}
Ve_i=\{(x_1,x_2,\ldots,x_n      )e_i;(x_1,x_2,\ldots,x_n)      \in V \}, \mbox{for} \ i \in \{1,2\}.
\end{eqnarray*}
Furthermore, the set $Ve_i$ is a subset of $\C_2^n$.
\end{definition}

\begin{definition}({\bf $e_1$-matrix})\label{th12}
If $A=[\xi_{ij}]_{n\times m}\in \C_2^{n\times m}$ and $\xi_{ij}={1_\xi}_{ij}e_1$ for all $i$  and $j$, then $A$ is called $e_1$-matrix. 
\end{definition}

\begin{definition}({\bf $e_2$-matrix})\label{th13}
If $A=[\xi_{ij}]_{n\times m}\in \C_2^{n\times m}$ and $\xi_{ij}={2_\xi}_{ij}e_2$ for all $i$ and $j$,  then $A$ is called $e_2$-matrix. 
\end{definition}

\begin{definition}({\bf $e_1e_2$-matrix})\label{th14}
If $A=[\xi_{ij}]_{n\times m}\in \C_2^{n\times m}$ and $\xi_{ij}\in I_1\cup I_2$ for all $i$ and $j$, then $A$ is called $e_1e_2$- matrix. 
\end{definition}

\begin{remark}
Every $e_1$-matrix is an $e_1e_2$-matrix, also every $e_2$-matrix is an $e_1e_2$-matrix, but converse need not be true. For example 
\[
A =\begin{bmatrix}
e_1 & e_2\\
e_2 & e_1
\end{bmatrix}
\] is $e_1e_2$-matrix, but neither $A$ is  $e_1$-matrix nor $A$ is  $e_2$-matrix. 
\end{remark}

\begin{theorem}\label{th15}
If $B$ is a sub-matrix of a bicomplex matrix $A$, then $1_B$ and $2_B$ are sub-matrices of $1_A$ and $2_A$ respectively.    
\end{theorem}
\begin{proof}
Let  $B$ be a sub-matrix of a bicomplex matrix $A$. Therefore, $B$ is obtained from $A$ by omitting some rows and columns of $A$. Thus, we have sub-matrices $1_B$ and $2_B$, which are obtained from matrix $B$ by omitting some rows and some columns of $1_A$ and $2_A$ respectively. Hence, $1_B$ and $2_B$ are sub-matrices of $1_A$ and $2_A$ respectively.
\end{proof}

\begin{remark}
The converse of Theorem \ref{th15} is not true in general. Consider a bicomplex matrix
\begin{align*} 
A = \begin{bmatrix}
1 & e_1 & 0 \\
e_1 & e_2 & e_1\\
0 & 0 & 1 
\end{bmatrix}
=\begin{bmatrix}
1 & 1 & 0 \\
1 & 0 & 1 \\
0 & 0 & 1
\end{bmatrix} e_1 + \begin{bmatrix}
1 & 0 & 0 \\
0 & 1 & 0 \\
0 & 0 & 1
\end{bmatrix} e_2. 
\end{align*}
We see that
\[\begin{bmatrix}
 1  & 1\\
 0 & 0
\end{bmatrix} 
\quad and \quad
\begin{bmatrix}
 1  & 0\\
 0 & 1
\end{bmatrix}
\]    
are sub-matrices of $1_A$ and $2_A$, respectively. But
\begin{align*} 
B =\begin{bmatrix}
1 & 1\\
0 & 0
\end{bmatrix} e_1 + \begin{bmatrix}
1 & 0\\
0 & 1
\end{bmatrix} e_2
=\begin{bmatrix}
1 & e_1\\
0 & e_2
\end{bmatrix}
\end{align*}
is not a sub-matrix of $A$.
\end{remark}

\begin{definition}({\bf Rank of a bicomplex matrix})\label{th17}
A number $r$ is said to be the rank of a matrix $A\in \C_2^{n\times m}$, if it possesses the following properties:
\begin{itemize} 
\item
There exists a non-singular sub-matrix $B_{r}$ of $A$  and a series of non-singular sub-matrices of the matrix $B_{r}$ such that
\begin{equation}\label{eq1}
B_1 \preccurlyeq B_2\preccurlyeq \ldots \preccurlyeq B_{r-1} \preccurlyeq B_{r},
\end{equation}
where $B_i \preccurlyeq B_{i+1}$ represents that $B_i$ is a sub-matrix of  $B_{i+1}$, and $i$ represents the order of the matrix $B_i$.
\item
If $B'$ is any sub-matrix of $A$ with order $t; t>r$, then no series of the type (\ref{eq1}) exists  for the  matrix $B'$.
\end{itemize}
\end{definition}
The rank of a matrix A is denoted by $\rho(A)$. Furthermore,
If no such $r$ exists, then we define $\rho(A)=0$.

\begin{remark} 
The definition \ref{th17} yields that
\begin{enumerate}
\item
 If $r$ is the rank of a bicomplex matrix $A$, then there exists a non-singular matrix $B$ of order $r$ such that $B\preccurlyeq A$ and $\rho(B)=r$, and any matrix $B'$ of order $t; t> r$ such that $B'\preccurlyeq A$, then $\rho(B') \leq r$.
\item 
If $A\in \C_2^{n\times m}$, then $\rho(A) \le$ min$(m,n)$.
\item 
If $A\in \C_2^{n\times n}$ and $\rho(A) = n$, then $\det(A) \notin O_2$. But converse need not be true. For example, if
$A =\begin{bmatrix}
e_1 & e_2\\
e_2 & e_1
\end{bmatrix}$,
then $\rho(A) = 0$, but $\det(A) \notin O_2$.
\item 
If $A\in \C_2^{n\times m}$ and $\rho(A) = 0$, then $A$ need not be null matrix. For example, the rank of $e_1e_2$-matrix is zero, but it is not a null matrix.
\end{enumerate}
\end{remark}

\begin{theorem}
Let $A$ be a bicomplex matrix of order $n\times m$. Then, $\rho(A)=0$ if and only if $A$ is $e_1e_2$- matrix.
\begin{proof}
Let $A=[\xi_{ij}]_{n\times m} \in \C_2^{n \times m}$, and $\rho(A)=0$.
Let us assume that there exists an entry $\xi_{ps}$ of $A$ such that $\xi_{ps} \notin O_2$ for some $p \in \{1,2,...,n\}$ and $s \in \{1,2,...,m\}$.
Therefore the matrix $B=[\xi_{ps}]_{1\times 1}$ is a non-singular sub-matrix of $A$ and has a series of the type (\ref{eq1}). There are two cases here:\\
\textbf{Case (i):} If no series of the type (\ref{eq1}) exists  for all sub-matrix $B'$ of $A$  with order $t; t>1$, then $\rho(A)=1$.\\
\textbf{Case (ii):} If there is a non-singular sub-matrix $B'$ of $A$ with order $t; t>1$ such that a series of the type (\ref{eq1}) exists  for the matrix $B'$, then $\rho(A) \geq t$.\\
Clearly, in each case $\rho(A) \geq 1$, which contradicts as $\rho(A)=0$. So, our assumption, $\xi_{ps} \notin O_2$ for some $p \in \{1,2,...,n\}$ and $s \in \{1,2,...,m\}$, is not true. Thus, $\xi_{ps} \in O_2$ for all $p \in \{1,2,...,n\}$ and $s \in \{1,2,...,m\}$. Hence $A$ is $e_1e_2$- matrix.\\
\noindent{\bf Conversely} Suppose $A=[\xi_{ij}]_{n\times m}$ is an $e_1e_2$- matrix. Let us assume that $\rho(A)=r,r \geq 1$.
Then there exists a non-singular sub-matrix $B_{r}$ of $A$ with order $r; r \geq 1$ such that a series $B_1 \preccurlyeq B_2\preccurlyeq \ldots \preccurlyeq B_{r-1} \preccurlyeq B_{r},$ of the type (\ref{eq1}) exists  for the matrix $B_{r}$.
Clearly, $B_1$ is a non-singular sub-matrix of $A$ with order 1. This implies that there exists an entry $\xi_{ps}$ of $A$ such that $\xi_{ps} \notin O_2$ for some $p \in \{1,2,...,n\}$ and $s \in \{1,2,...,m\}$, which contradicts as $A$ is $e_1e_2$- matrix. Therefore our assumption $\rho(A)=r,r \geq 1$ is not true.
Hence $\rho(A)=0$.
\end{proof}
\end{theorem}

\begin{definition}({\bf Row rank of a matrix})\label{th19}
 If $A=[\xi_{ij}]_{n\times m} \in \C_2^{n \times m}$, then the row space of $A$ is defined as the subspace of vector-space $\C_2^m(\C_1)$ spanned by the rows of $A$. The
 row rank of $A$ is denoted by "$\rho_r(A)$" and defined as the dimension of row space of $A$.  
\end{definition} 
\begin{definition}({\bf Column rank of a matrix})\label{th20}
 If $A=[\xi_{ij}]_{n\times m} \in \C_2^{n \times m}$, then the column space of $A$ is defined as the subspace of vector-space $\C_2^n(\C_1)$ spanned by the columns of $A$. The
 column rank of $A$ is denoted by "$\rho_c(A)$" and defined as the dimension of column space of $A$.
\end{definition}

\begin{theorem} \label{th18}
Let $A$ be a bicomplex matrix of order $n\times m$. Then, 
\begin{center}
$\rho(A)\le min \ (\rho(1_A),\rho(2_A))$.
\end{center}
\end{theorem}
\begin{proof}
Let $A$ be a bicomplex matrix of order $n\times m$ and $\rho(1_A)=r_1$, $\rho(2_A)=r_2$. Let us consider $\rho(A)=l$ and  $l>min(r_1,r_2)$. Then, there exists a sub-matrix $B$ of $A$ of order $l\times l$ such that $\rho(B)=l$. This implies that $\det(B)\notin O_2$.
\begin{eqnarray*}
&\Rightarrow& \det(1_B)\neq 0 \ \mbox{and} \ \det(2_B)\neq 0\ \{\mbox{by Corollary\ref{th9}}\}\\
&\Rightarrow& \rho(1_B) = l \ \mbox{and} \ \rho(2_B) = l\\
&\Rightarrow& \rho(1_B) > r_1 \ \mbox{or} \ \rho(2_B) > r_2 \quad \{\mbox{as} \ l>min(r_1,r_2)\}\\
&\Rightarrow& \rho(1_B) > \rho(1_A) \ \mbox{or} \ \rho(2_B) > \rho(2_A),
\end{eqnarray*}
which is a contradiction, because $(1_B)$ and $(2_B)$ are the sub-matrices of $(1_A)$ and $(2_A)$ respectively. Therefore, our assumption  $\rho(A)=l$, and  $l>min(r_1,r_2)$ is not true.
Thus, $\rho(A) \leq min(r_1,r_2)$. Hence, $\rho(A) \leq  min(\rho(1_A),\rho(2_A))$.
\end{proof}

The following remark is immediate consequence of Theorem \ref{th18}.

\begin{remark} If $A$ is a bicomplex matrix of order $n\times m$, then
\begin{enumerate}
\item 
$\rho(A)\leq \rho(1_A)$ and $\rho(A)\leq \rho(2_A)$. 
\item
$\rho(A)\leq$ $\rho_r(1_A)$ and $\rho(A)\leq$ $\rho_r(2_A)$.
\item 
$\rho(A)\leq$ $\rho_c(1_A)$ and $\rho(A)\leq$ $\rho_c(2_A)$.
\end{enumerate}
\end{remark}

\begin{theorem}\label{th21}
A matrix $A\in \C_2^{n\times n}$ is non-singular if and only if $\rho(1_A) = n$  and $\rho(2_A) = n$. 
\end{theorem}

\begin{proof}
Let $A\in \C_2^{n\times n}$ be a non-singular matrix. Then, by using Corollary \ref{th9} we have
\begin{eqnarray*}
\det(1_A) \ne 0 \ \mbox{and} \ \det(2_A)\neq 0
\Rightarrow \rho(1_A) = n \ \mbox{and} \ \rho(2_A) = n.    
\end{eqnarray*}

\noindent{\bf Conversely} Let $\rho(1_A)=n$ and $\rho(2_A)=n$, where $1_A,\ 2_A\in \C_1^{n\times n}(\C_1)$. Then,
\begin{eqnarray*}
\det(1_A)\neq0 \ \mbox{and} \ \det(2_A)\neq0.
\end{eqnarray*}
Using Corollary \ref{th9}, we get $A$ is non singular matrix. The proof of theorem is completed.
\end{proof}

The following corollary is immediate consequence of Theorem \ref{th21}.

\begin{corollary}\label{th22}    
If $A\in \C_2^{n\times m}$,  and  $\rho(A)=n($respectively $m)$, then $\rho(1_A)=n($respectively $m)$ and  $\rho(2_A) =n($respectively $m)$.    
\end{corollary}

\begin{remark}
The converse of the Corollary \ref{th22} is not true. For instance, if
\begin{align*} 
A =\begin{bmatrix}
1 & 0 & 0\\
0 & 1 & 6
\end{bmatrix} e_1 + \begin{bmatrix}
0 & 1& 0\\
1 & 0 & 6
\end{bmatrix} e_2
=\begin{bmatrix}
e_1 & e_2 & 0\\
e_2 & e_1 & 6
\end{bmatrix},
\end{align*}
then $\rho(1_A)=\rho(2_A)=2$, but $\rho(A)=1$.
\end{remark}

\begin{remark}
If $A\in \C_2^{n\times m}$, then in general $\rho_r(A) \neq$ $\rho_c(A)$. 
\end{remark}

\begin{example}\label{th23}
Consider \[
 A=\begin{bmatrix}
2 & e_1 & 1\\
0 & 0 & e_1
\end{bmatrix}.
\]
It is evident that columns of A and rows of A are linearly independent with complex scalars. Thus, 
$\rho_r(A) \neq$ $\rho_c(A)$.
\end{example}

\begin{remark}
Example \ref{th23} shows that a bicomplex matrix $A\in \C_2^{n\times m}$ does not exhibit analogous notion as observed in a complex matrix, i.e. all columns of a bicomplex matrix $A$ are linearly independent or all rows of a bicomplex matrix $A$ are linearly independent  does not imply $\rho(A)=m$ or $n$.     
\end{remark}

\begin{remark}
If $A\in\C_2^{n\times n}$, and $\det(A)\in O_2$, then rows of A and columns of A may be linearly independent.
\end{remark}

\begin{example}\label{th24}
Consider \[
 A=\begin{bmatrix}
1 & 2\\
e_1 & 0
\end{bmatrix}.
\]
The rows and columns of matrix A are linearly independent, i.e. $\rho_r(A) =\rho_c(A)=2$, but $\det(A)\in O_2$.
\end{example}

\begin{remark}
If $A\in \C_2^{n\times m}$, then in general
$\rho_r(A) \neq \rho_c(A) \neq \rho(A)$. 
\end{remark}
\begin{example}\label{th25}
Consider \[
A=\begin{bmatrix}
1 & 0 & 1\\
e_2 & e_1 & 0\\
0 & e_2 & e_1\\
0 & 0 & e_1
\end{bmatrix}.
\]
Then, we have $\rho_r(A)=4$, $\rho_c(A)=3$ and $\rho(A)=1$, i.e. $\rho_r(A)\neq$ $\rho_c(A)\neq \rho(A)$.
\end{example}
\begin{definition}\label{th26}
({\bf Idempotent row rank of a matrix}) If $A = [\xi_{ij}]_{n\times m}\in \C_2^{n\times m}$,
then we define idempotent row space of $A$ as
\begin{eqnarray*}
\mbox{Idempotent row space of}\ A &=& (\mbox{row space of}\ 1_A) e_1  + (\mbox{row space of}\ 2_A) e_2\\
&=& \left\{  \left(\sum_{k=1}^{n}\alpha_k({1_\xi}_{k1},{1_\xi}_{k2},\ldots,{1_\xi}_{km})\right)e_1 \right. \\
  &&\left.   +\left(\sum_{k=1}^{n}\beta_k({2_\xi}_{k1},{2_\xi}_{k2},\ldots,{2_\xi}_{km})\right)e_2;\ \alpha_k,\beta_k\in \C_1, 1\le k\le n \right\}.
\end{eqnarray*}
The Idempotent row rank of $A$ is denoted by "$\rho_{ir}(A)$" and defined as the dimension of idempotent row space of $A$.
\end{definition}

\begin{definition}({\bf Idempotent column rank of a matrix})\label{th27}
If $A = [\xi_{ij}]_{n\times m}\in \C_2^{n\times m}$, then we define idempotent column space of $A$ as 
\begin{eqnarray*}
\mbox{Idempotent column space of}\ A&=&(\mbox{column space of}\ 1_A) e_1  + (\mbox{column space of}\ 2_A) e_2\\
&=& \left\{  \left(\sum_{k=1}^{m}\alpha_k({1_\xi}_{1k},{1_\xi}_{2k},\ldots{1_\xi}_{nk})\right)e_1 \right. \\
  &&\left.   +\left(\sum_{k=1}^{m}\beta_k({2_\xi}_{1k},{2_\xi}_{2k},\ldots,{2_\xi}_{nk})\right)e_2;\ \alpha_k,\beta_k\in \C_1, 1\le k\le m \right\}.
\end{eqnarray*}
The Idempotent column rank of $A$ is denoted by "$\rho_{ic}(A)$" and defined as the dimension of idempotent column space of $A$.
\end{definition}

\begin{theorem}\label{th28}
If $A=[\xi_{ij}]_{n\times m}\in \C_2^{n\times m}$, then row space of \ $A \subseteq$ Idempotent row space of \ $A$.
\end{theorem}

\begin{proof}
Let $A=[\xi_{ij}]_{n\times m}\in \C_2^{n\times m}$ and $X \in$ row space of $A$. Then, 
$X = \sum_{k=1}^{n}\alpha_{k}(\xi_{k1},\xi_{k2},\ldots,\xi_{km})$, where $\alpha_k \in \C_1$. Now, by using Definition \ref{th5} and Remark \ref{th49}, we have 
\begin{eqnarray*}
X &=& \sum_{k=1}^{n}\alpha_{k}\{({1_\xi}_{k1},{1_\xi}_{k2},\ldots,{1_\xi}_{km}) e_1 + ({2_\xi}_{k1},{2_\xi}_{k2},\ldots,{2_\xi}_{km}) e_2\}\\
&=& \sum_{k=1}^{n}\alpha_{k}\{({1_\xi}_{k1}e_1,{1_\xi}_{k2}e_1,\ldots,{1_\xi}_{km}e_1) +({2_\xi}_{k1}e_2,{2_\xi}_{k2}e_2,\ldots,{2_\xi}_{km}e_2)\}\\
&=& \sum_{k=1}^{n} \alpha_{k}({1_\xi}_{k1} e_1,{1_\xi}_{k2} e_1,\ldots,{1_\xi}_{km} e_1) + \sum_{k=1}^{n}\alpha_{k}({2_\xi}_{k1} e_2,{2_\xi}_{k2} e_2,\ldots,{2_\xi}_{km} e_2)\\
&=& \sum_{k=1}^{n}(\alpha_{k} ({1_\xi}_{k1} e_1),\alpha_{k}({1_\xi}_{k2} e_1),\ldots,\alpha_{k}({1_\xi}_{km} e_1)) \\
&&\hspace{3.5 cm}+ \sum_{k=1}^{n}(\alpha_{k}({2_\xi}_{k1} e_2),\alpha_{k}({2_\xi}_{k2} e_2),\ldots,\alpha_{k}({2_\xi}_{km} e_2))\\
&=& \sum_{k=1}^{n}((\alpha_{k} {1_\xi}_{k1}) e_1,(\alpha_{k}{1_\xi}_{k2}) e_1,\ldots,(\alpha_{k}{1_\xi}_{km} )e_1)\\
&&\hspace{3.5 cm}+\sum_{k=1}^{n}((\alpha_{k}{2_\xi}_{k1}) e_2,(\alpha_{k}{2_\xi}_{k2}) e_2,\ldots,(\alpha_{k}{2_\xi}_{km}) e_2)\\ 
&=& \sum_{k=1}^{n}(\alpha_{k} {1_\xi}_{k1} ,\alpha_{k}{1_\xi}_{k2},\ldots,\alpha_{k}{1_\xi}_{km} )e_1 + \sum_{k=1}^{n}(\alpha_{k}{2_\xi}_{k1},\alpha_{k}{2_\xi}_{k2},\ldots,\alpha_{k}{2_\xi}_{km})e_2\\
&=& \sum_{k=1}^{n}\{\alpha_{k} ( {1_\xi}_{k1} ,{1_\xi}_{k2},\ldots,{1_\xi}_{km} )\}e_1+ \sum_{k=1}^{n}\{\alpha_{k}({2_\xi}_{k1},{2_\xi}_{k2},\ldots,{2_\xi}_{km})\}e_2 \\
&=& \left \{\sum_{k=1}^{n}\alpha_{k} ( {1_\xi}_{k1} ,{1_\xi}_{k2},\ldots,{1_\xi}_{km} )\right\}e_1+\left \{\sum_{k=1}^{n}\alpha_{k}({2_\xi}_{k1},{2_\xi}_{k2},\ldots,{2_\xi}_{km})\right \}e_2.
\end{eqnarray*}
Therefore, $X \in$ (row space of $1_A )  e_1  +$ (row space of $2_A )   e_2$. Hence, row space of $A  
\subseteq$ idempotent row space of $A$.
\end{proof}
\noindent{Dually, we may prove the following theorem.}
\begin{theorem} \label{th29}
If $A=[\xi_{ij}]_{n\times m}\in \C_2^{n\times m}$, then 
column  space  of  $A  \subseteq$ idempotent column  space  of $A$.
\end{theorem}
%$\left[ row-rank \left( 1_A \right)=n \right]$%
\begin{theorem} \label{th30}
Let $A=[\xi_{ij}] \in \C_2^{n\times m}$ and either rows of $1_A$  or rows of $2_A$ are linearly independent. Then, $\rho_r(A)$ = n, i.e. rows of A are linearly independent. 
\end{theorem}
\begin{proof}
Let $A=[\xi_{ij}]_{n\times m}\in \C_2^{n\times m}$ and rows of $1_A$ be linearly independent.\\ 
Suppose
$\sum_{k=1}^{n}\alpha_k (\xi_{k1},\xi_{k2},\dots,\xi_{km})=0$, where $\alpha_k \in \C_1$. Then, by Definition \ref{th5} and Remark \ref{th49}, it follows that
\begin{eqnarray*}
&& \sum_{k=1}^{n} \alpha_k \{({1_\xi}_{k1},{1_\xi}_{k2},\dots,{1_\xi}_{km}) e_1 + ({2_\xi}_{k1},{2_\xi}_{k2},\dots,{2_\xi}_{km}) e_2 \} = (0,0,\dots,0)\\
&\Rightarrow& \sum_{k=1}^{n} \alpha_k \{({1_\xi}_{k1}e_1, {1_\xi}_{k2}e_1, \dots,{1_\xi}_{km}e_1) + ({2_\xi}_{k1}e_2, {2_\xi}_{k2}e_2, \dots,{2_\xi}_{km}e_2)\}= (0,0,\dots,0) \\
&\Rightarrow&\sum_{k=1}^{n} \alpha_k({1_\xi}_{k1}e_1, {1_\xi}_{k2}e_1, \dots,{1_\xi}_{km}e_1) + \sum_{k=1}^{n} \alpha_k ({2_\xi}_{k1}e_2, {2_\xi}_{k2}e_2, \dots,{2_\xi}_{km}e_2)= (0,0,\dots,0)\\
&\Rightarrow&\sum_{k=1}^{n} (\alpha_k({1_\xi}_{k1}e_1), \alpha_k({1_\xi}_{k2}e_1),\dots, \alpha_k({1_\xi}_{km}e_1)) \\
&&\hspace{3.0 cm} +\sum_{k=1}^{n} (\alpha_k ({2_\xi}_{k1}e_2), \alpha_k({2_\xi}_{k2}e_2),\dots, \alpha_k({2_\xi}_{km}e_2)) = (0,0,\dots,0)\\
&\Rightarrow&\sum_{k=1}^{n} ((\alpha_k{1_\xi}_{k1})e_1, (\alpha_k{1_\xi}_{k2})e_1,\dots, (\alpha_k{1_\xi}_{km})e_1)
\\ 
&&\hspace{3.0 cm}+\sum_{k=1}^{n} ((\alpha_k {2_\xi}_{k1})e_2, (\alpha_k{2_\xi}_{k2})e_2,\dots, (\alpha_k{2_\xi}_{km})e_2)= (0,0,\dots,0)\\
&\Rightarrow&\sum_{k=1}^{n} (\alpha_k{1_\xi}_{k1}, \alpha_k{1_\xi}_{k2},\dots, \alpha_k{1_\xi}_{km})e_1\\ &&\hspace{3.0 cm}+\sum_{k=1}^{n} (\alpha_k {2_\xi}_{k1}, \alpha_k{2_\xi}_{k2},\dots, \alpha_k{2_\xi}_{km})e_2 = (0,0,\dots,0)\\
&\Rightarrow& \left\{\sum_{k=1}^{n} (\alpha_k{1_\xi}_{k1}, \alpha_k{1_\xi}_{k2},\dots, \alpha_k{1_\xi}_{km})\right\}e_1\\
&&\hspace{3.0 cm}+\left\{\sum_{k=1}^{n} (\alpha_k {2_\xi}_{k1}, \alpha_k{2_\xi}_{k2},\dots, \alpha_k{2_\xi}_{km})\right\}e_2 = (0,0,\dots,0).
\end{eqnarray*}
Implies that,
\[
\sum_{k=1}^{n} \alpha_k ({1_\xi}_{k1},{1_\xi}_{k2},\dots,{1_\xi}_{km}) = (0,0,\dots,0)\quad and \quad \sum_{k=1}^{n} \alpha_k ({2_\xi}_{k1},{2_\xi}_{k2},\dots,{2_\xi}_{km}) =(0,0,\dots,0).
\]
Since rows of $1_A$ are linearly independent,
therefore $\alpha_{k} =0\ \forall\ k; 1\le k\le n$. 
So, it is clear that rows of $A$ are linearly independent.
Similarly, if rows of $2_A$ are linearly independent.
It gives rows of $A$ are linearly independent. Hence $\rho_r(A) = n$.
\end{proof}

\begin{remark} 
Some remarks can be easily made here.
\begin{enumerate}
    \item 
The converse of Theorem \ref{th30} is not true. 
For example, consider 
\[
 A=\begin{bmatrix}
0 & 0 & e_1\\
0 & 0 & e_2
\end{bmatrix}.
\]
Then, we have $\rho_r(A)$ = 2. But, rows of $1_A$ are linearly dependent and $2_A$ are linearly dependent.
\item 
If $A\in \C_2^{n\times m}$ and $\rho(A) = n$, then Theorem \ref{th30} and Corollary \ref{th22} implies that
$\rho_r(A)$ = $\rho(A)$ = $\rho_r(1_A)$ = $\rho_r(2_A) = n$. 
\end{enumerate}
\end{remark}

The following Corollary \ref{th31} is immediate consequence of Theorems \ref{th21} and \ref{th30}.

\begin{corollary}\label{th31}
Let $A=[\xi_{ij}]_{n\times n} \in \C_2^{n\times n}$ and $\det(A)\notin O_2$, then $\rho_r(1_A)$ = $\rho_r(2_A)$ = $\rho_r(A)=n$.
\end{corollary}
Dually we may prove the following result.
\begin{theorem}\label{th32}
If $A\in \C_2^{n\times m}$ and either columns of $1_A$ or columns of $2_A$ are linearly independent then $\rho_c(A)$ = m.
\end{theorem}

\begin{remark}
The following remarks can be made here.
\begin{enumerate}
\item 
The converse of the Theorem \ref{th32} is not true. For example, consider
 
\[
 A=\begin{bmatrix}
e_1 & 0 \\
0 & e_2\\
\end{bmatrix}.
\]
It is trivial to find that $\rho_c(A)$ = 2, but columns of $1_A$ are linearly dependent and  columns of $2_A$ are linearly dependent.
\item 
If $A\in \C_2^{n\times m}$ and $\rho(A)=m$, then Theorem \ref{th32} and Corollary \ref{th22} imply that
$\rho_c(A)$ = $\rho(A)$ = $\rho_c(1_A)$ = $\rho_c(2_A) = m$.
\end{enumerate}
\end{remark}
The following corollary is immediate consequence of Theorems \ref{th21} and \ref{th32}.
\begin{corollary}\label{th33}
If $A=[\xi_{ij}]_{n\times n} \in \C_2^{n\times n}$ and  $\det(A) \notin O_2$, then $\rho_c(1_A)$ = $\rho_c(2_A)$ = $\rho_c(A) = n$.
\end{corollary}
The following Corollaries \ref{th34} and \ref{th35} are immediate consequences of Theorems \ref{th30} and \ref{th32}.
\begin{corollary}\label{th34}
If $A=[\xi_{ij}]_{n \times n} \in \C_2^{n\times n}$ and either rows of $1_A$ (respectively $2_A$) or columns of $1_A$ (respectively $2_A$) are linearly independent, then $\rho_r(A)$ = $\rho_c(A) = n$.      
\end{corollary}
    
\begin{corollary}\label{th35}
If $A\in \C_2^{n\times n}$ and $det(A)\notin O_2$, then $\rho_r(A)$ = $\rho_c(A) = n$.   
\end{corollary}

\begin{proposition}\label{th36}
If $A=[\xi_{ij}]_{n\times n} \in \C_2^{n\times n}$ and $\rho(A) = n$, then $\rho(A)$ = $\rho_c(A)$ = $\rho_r(A)$.    
\end{proposition}
\begin{proof}
Let $A=[\xi_{ij}]_{n \times n} \in \C_2^{n\times n}$ and $\rho(A) = n$, then $\det(A)\notin O_2$. Thus by using corollary \ref{th35},  $\rho(A)$ = $\rho_r(A)$ = $\rho_c(A)$.  
\end{proof}
\begin{theorem}\label{th37}
Let $A=[\xi_{ij}]_{n\times n} \in \C_2^{n\times n}$ be a non-singular matrix such that $\xi_{ij}\notin O_2$ for all $i,j$ and $n\le 3$. Then, $\rho(A) = n$.  
\end{theorem}
\begin{proof}
Let $A=[\xi_{ij}]_{n\times n} \in \C_2^{n\times n}$ be  non-singular matrix such that $\xi_{ij}\notin O_2$ for all $i,j$.\\  
\textbf{Case (i):} If $n=1$, then clearly $\rho(A) = 1$.\\
\textbf{Case (ii):} If $n=2$, then 
$A =\begin{bmatrix}
\xi_{11} & \xi_{12}\\
\xi_{21} & \xi_{22}\\
\end{bmatrix}$.\\
Consider 
\begin{align*}
A_1=
\begin{bmatrix}
\xi_{11}
\end{bmatrix}\quad  \mbox{and} \quad A_2=\begin{bmatrix}
\xi_{11} & \xi_{12} \\
\xi_{21} & \xi_{22}  
\end{bmatrix}.   
\end{align*}
Since $\xi_{ij}\notin O_2$ for all $i,j$ and $\det(A) \notin O_2$, therefore we have a series of non-singular sub-matrices of $A$ such that $A_1 \preccurlyeq A_2=A$. Thus $\rho(A)$ will be $2$.\\
\textbf{Case (iii):} $n=3$.\\
Let, $B=1_B e_1 +2_B e_2$ be any sub-matrix of $A$ of order $2\times 3$. Then
\begin{align*}
1_B=\begin{bmatrix}
C_1 & C_2 & C_3    
\end{bmatrix} \quad \mbox{and} \quad
2_B=\begin{bmatrix}
C_1' & C_2' & C_3'
\end{bmatrix},
\end{align*}
where $C_i,C_i';i\in\{1,2,3\}$ are columns of $1_B$ and $2_B$ respectively.
Since $\det(A) \notin O_2$ and $\xi_{ij}\notin O_2$ for all $i,j$, therefore $\rho_c(1_B) = \rho_c(2_B) = 2$ and $C_i\neq 0,C_i'\neq 0;\forall \ i\in\{1,2,3\}$.\\
Since $\rho_c(1_B) = 2$ , this implies that there exist $i, j\in\{1,2,3\};i<j$ such that $C_i$ and $C_j$ are linearly independent.\vspace{0.05cm}

Let us suppose \textquotedblleft $C_i \ \mbox{and} \ C_k$\textquotedblright \ and \textquotedblleft$C_j\ \mbox{and} \ C_k$\textquotedblright \ be linearly dependent columns, where $k\in\{1,2,3\};i\neq k$ and $j\neq k$. Then there exist  $\alpha\neq0$ and $\beta\neq0 \in C_1$ such that $C_i=\alpha C_k$ and $C_j=\beta C_k$. It follows that $C_i$ and $C_j$ are linearly dependent, which is a contradiction because $C_i$ and $C_j$ are linearly independent. Therefore, our assumption \textquotedblleft$C_i \ \mbox{and} \ C_k$\textquotedblright \ and \textquotedblleft$C_j\ \mbox{and} \ C_k$\textquotedblright are linearly dependent is not true, i.e. either \textquotedblleft$C_i \ \mbox{and} \ C_k$\textquotedblright are linearly independent or \textquotedblleft$C_j\ \mbox{and} \ C_k$\textquotedblright are linearly independent, which evidently implies that the set of vectors $\{C_i,C_k\}$ and $\{C_i,C_j\}$ are linearly independent or the set of vectors $\{C_i,C_j\}$ and $\{C_j,C_k\}$ are linearly independent for $i,j \ \mbox{and} \ k \in \{1,2,3\}$; $i\neq j, j\neq k$ and $k\neq i$.

In same manner, for the matrix $2_B$ we can find that either the set of vectors $\{C_l',C_m'\}$ and $\{C_l',C_n'\}$ or the set of vectors $\{C_l',C_m'\}$ and $\{C_m',C_n'\}$ are linearly independent.

Now, we list the possible pair of non-singular sub-matrices of $1_B$ and $2_B$ of order $2\times2$:
\begin{center}
\begin{tabular}{||c | c||} 
 \hline
 \textbf{Possible pair of non singular} &  \textbf{Possible pair of non singular} \\ 
  \textbf{sub-matrices of $1_B$ of order 2} &  \textbf{sub-matrices of $2_B$ of order 2} \\
 [0.5ex] 
 \hline\hline
 $[C_1,C_2]$ and $[C_1,C_3]$ & $[C_1',C_2']$ and $[C_1',C_3']$  \\ 
 \hline
 $[C_1,C_2]$ and $[C_2,C_3]$ & $[C_1',C_2']$ and $[C_2',C_3']$  \\
 \hline
 $[C_1,C_3]$ and $[C_2,C_3]$ & $[C_1',C_3']$ and $[C_2',C_3']$ \\
 \hline
\end{tabular}
\end{center}
It is completely clear from the above that there are nine cases formed from the above possibilities to form non singular sub-matrix of $B$ of order $2$. In each case, there exist $\alpha,\beta$ such that $[C_\alpha, C_\beta] \ and \ [C'_\alpha, C'_\beta] $ are non singular sub-matrices of order $2$ of $1_B$ and $2_B$ respectively for some $\alpha,\beta \in \{1,2,3\}$, where $\alpha<\beta$.

Now, we construct a bicomplex matrix 
\[
M=[C_\alpha \ C_\beta]e_1+[C'_\alpha \ C'_\beta]e_2.
\]

It is evident that $M$ is a non-singular sub-matrix of order $2$ of matrix $B$. This implies that $M$ is the sub-matrix of order $2$ of matrix $A$ such that $\det(M)\notin O_2$. Then, there exist $\xi_{pq}\notin O_2$ such that $[\xi_{pq}]_{1\times 1}$ is a sub-matrix of $M$ of order $1$, where $p,q\in \{1,2,3\}$.

Let $A_1=[\xi_{pq}]_{1 \times 1}$, $A_2=M$ and $A_3=A$. Since $\xi_{pq}\notin O_2$, $\det(M)\notin O_2$ and $\det(A)\notin O_2$, therefore we have a series of non- singular sub-matrices of A such that
\[
A_1\preccurlyeq A_2 \preccurlyeq A_3=A
.\]
Hence $\rho(A) = 3$.\\ 
The proof of theorem is completed.
\end{proof}

\begin{theorem}\label{th38}
Let $V$ be a subspace of $\C_1^n(\C_1)$. Then, $V e_1$ is a subspace of $\C_2^n(\C_1)$.
\end{theorem}
\begin{proof}
Let $V$ be a subspace of $\C_1^n(\C_1)$. Then, $Ve_1$ is a subset of $\C_2^n$.\\
Let us consider $X,Y\in Ve_1$. Then, there exists $(x_1,x_2,\ldots,x_n),(y_1,y_2,\ldots,y_n)\in V$ such that $X=(x_1,x_2,\ldots,x_n)e_1$ and $Y=(y_1,y_2,\ldots,y_n)e_1$. Now, by using Remark \ref{th49}, we have
\begin{eqnarray*}
X+Y &=& (x_1,x_2,\ldots,x_n)e_1+(y_1,y_2,\ldots,y_n)e_1\\
&=&(x_1 e_1,x_2 e_1,\ldots,x_n e_1)+(y_1 e_1,y_2
e_1,\ldots,y_n e_1)\\
&=&(x_1 e_1+y_1 e_1,x_1 e_1+y_1 e_1,\ldots,x_n e_1+y_n e_1)\\
&=&((x_1+y_1)e_1,(x_2+y_2)e_1,\ldots,(x_n+y_n)e_1)\\
&=&(x_1+y_1, x_2+y_2, \ldots,x_n+y_n)e_1\in V e_1
.\end{eqnarray*}
and     
\begin{eqnarray*}
\alpha X &=& \alpha\{(x_1,x_2,\ldots,x_n)e_1\}\\
&=&\alpha(x_1 e_1,x_2 e_1,\ldots,x_n e_1)\\
&=&(\alpha(x_1 e_1),\alpha(x_2 e_1),\ldots,\alpha(x_n e_1))\\
&=&((\alpha x_1) e_1,(\alpha x_2) e_1,\ldots,(\alpha x_n) e_1)\\
&=&(\alpha x_1, \alpha x_2, \ldots, \alpha x_n) e_1 \in V e_1
.\end{eqnarray*}
Since $V e_1$ is a subset of $\C_2^n$. Thus, $V e_1$ forms a subspace of $\C_2^n(\C_1)$.
\end{proof}
Dually we may prove the following result.
\begin{theorem}\label{th39}
Let $V$ be a  subspace of $\C_1^n(\C_1)$. Then, $V e_2$ is subspace of $C_2^n(\C_1)$.    
\end{theorem}

Combining Theorems \ref{th38} and \ref{th39} we get following result.
\begin{proposition}\label{th40}
If $1_V$ and  $2_V$ are two sub-spaces of $\C_1^n(\C_1)$, then $1_V e_1 + 2_V e_2$ is a subspace of $\C_2^n(\C_1)$.    
\end{proposition}

\begin{remark}
It follows from Proposition \ref{th40}, if $A\in \C_2^{n\times m}$, then idempotent row space of $A$ and idempotent column space of $A$ are sub-spaces of $\C_2^m$ and $\C_2^n$, respectively.   
\end{remark} 
The following corollary is immediate consequence of Remark 2.42 and Theorem \ref{th29}.
\begin{corollary}\label{th41}
If $A\in \C_2^{n\times m}$, then column space of $A$ is a subspace of idempotent column space of $A$.   
\end{corollary}
The following corollary is immediate consequence of Remark 2.42 and Theorem \ref{th28}.
\begin{corollary}\label{th42}
If $A\in \C_2^{n\times m}$, then row space of $A$ is a subspace of idempotent row space of $A$.   
\end{corollary}

\begin{proposition}\label{th43}
If $1_V$ and $2_V$ be two sub-spaces of $\C_1^n$. Then $1_Ve_1\cap 2_V e_2=\{0\}$.
\end{proposition}
\begin{proof}
Let $1_V$ and $2_V$ be two sub-spaces of $\C_1^n$ and $(\xi_{1},\xi_{2},\cdots, \xi_{n}) \in 1_Ve_1  \cap 2_V e_2$. Then, 
\begin{center}
 $(\xi_{1},\xi_{2},\cdots, \xi_{n}) \in 1_V e_1$ and $(\xi_{1},\xi_{2},\cdots, \xi_{n}) \in 2_V e_2$.   
\end{center}
This implies that there exist \  $(z_{1},z_{2},\cdots, z_{n})\in 1_V$ and $( w_{1},w_{2},\cdots, w_{n})\in 2_V$ such that
\[
(\xi_{1},\xi_{2},\cdots, \xi_{n}) =(z_{1},z_{2},\cdots, z_{n})e_1 \mbox{and} (\xi_{1},\xi_{2},\cdots, \xi_{n}) =(w_{1},w_{2},\cdots, w_{n})e_2.
\]
Then, by using Remark \ref{th49}, we have
$(\xi_{1},\xi_{2},\cdots, \xi_{n}) =(z_{1}e_1,z_{2}e_2,\cdots, z_{n}e_1) \ and \ (\xi_{1},\xi_{2},\cdots, \xi_{n}) =(w_{1}e_2,w_{2}e_2,\cdots, w_{n}e_2)$
\begin{eqnarray*} 
&\Rightarrow&\xi_{i} = z_{i} e_1\ \ and\ \ \xi_{i} = w_{i}e_2 \ \forall \ i; \ 1\le i \le n\\   
&\Rightarrow& \ z_{i} e_1=w_{i} e_2 \ \forall \ i ;\ 1\le i \le n\\
&\Rightarrow& z_{i}=0 \quad and \quad w_{i}=0\ \forall \ i;\ 1\le i \le n\\
&\Rightarrow& (\xi_{1},\xi_{2},\cdots, \xi_{n})=(0,0,\cdots, 0)
.\end{eqnarray*}
Therefore,
\begin{align}\label{1}
1_V e_1 \cap 2_V e_2 \subseteq \{0\}.    
\end{align}
Since $(0,0,\cdots, 0)\in 1_V e_1$   and $(0,0,\cdots, 0) \in 2_V e_2$. Clearly $(0,0,\cdots, 0)\in 1_V e_2 \cap 2_V e_2$. So, we have
\begin{align}\label{2}
\{0\} \subseteq 1_V e_2 \cap 2_V e_2
.\end{align}
Using \eqref{1} and \eqref{2}, we get
\[
1_V e_2 \cap 2_V e_2 = \{0\}
.\]
The proof of proposition is completed.
\end{proof}
The following corollary is immediate consequence of Proposition \ref{th40} and Proposition \ref{th43}.
\begin{corollary}\label{th44}
If $1_V$ and  $2_V$ are two sub-spaces of $\C_1^n(\C_1)$. Then, $1_V e_1 + 2_V e_2 = 1_V e_1 \oplus 2_V e_2$.  
\end{corollary}

\begin{theorem}\label{th45}
Let $1_V$ be a  subspace of $\C_1^n(\C_1)$ and $\dim(1_V)=k$. Then, $\dim(1_V e_1)=k$.
\end{theorem}
\begin{proof}
Let $1_V$ be a subspace of $\C_1^n(\C_1)$ and $\dim(1_V)=k$.
So we take 
\begin{center}
 $S=\{(\beta_{11},\beta_{12},\ldots,\beta_{1n}), (\beta_{21},\beta_{22},\ldots,\beta_{2n}), \ldots, (\beta_{k1},\beta_{k2},\ldots,\beta_{kn})\}$   
\end{center}
as a basis for  $1_V$. Now, We claim that the set 
\begin{center}
 $S'=\{(\beta_{11},\beta_{12},\ldots,\beta_{1n})e_1, (\beta_{21},\beta_{22},\ldots,\beta_{2n})e_1, \ldots, (\beta_{k1},\beta_{k2},\ldots,\beta_{kn})e_1\}$   
\end{center}
forms a  basis for $1_Ve_1$. Suppose $\sum_{i=1}^{k} \alpha_i \{(\beta_{i1},\beta_{i2},\ldots,\beta_{in})e_1\} = 0, \alpha_i\in \C_1$. Using Remark \ref{th49}, we have
\begin{eqnarray*}
&&\sum_{i=1}^{k} \alpha_i (\beta_{i1}e_1,\beta_{i2}e_1,\ldots,\beta_{in}e_1) = 0\\ 
&\Rightarrow& \ \sum_{i=1}^{k}  (\alpha_i (\beta_{i1}e_1),\alpha_i(\beta_{i2}e_1),\ldots,\alpha_i(\beta_{in}e_1)) = 0\\ 
&\Rightarrow& \sum_{i=1}^{k}  ((\alpha_i \beta_{i1})e_1,(\alpha_i \beta_{i2})e_1,\ldots,(\alpha_i\beta_{in})e_1) =0\\
&\Rightarrow& \sum_{i=1}^{k}  (\alpha_i \beta_{i1},\alpha_i \beta_{i2},\ldots,\alpha_i\beta_{in})e_1 = 0\\
&\Rightarrow& \left\{\sum_{i=1}^{k} (\alpha_i \beta_{i1},\alpha_i \beta_{i2},\ldots,\alpha_i\beta_{in})\right\}e_1 = 0 \\
&\Rightarrow& \sum_{i=1}^{k} (\alpha_i \beta_{i1},\alpha_i \beta_{i2},\ldots,\alpha_i\beta_{in}) = 0\\
&\Rightarrow& \sum_{i=1}^{k} \alpha_i( \beta_{i1}, \beta_{i2},\ldots,\beta_{in}) = 0 
.\end{eqnarray*}
Since $S$ is a basis of $1_V$, implies that $\alpha_i=0 \ \forall i$. Therefore $S'$ is a linearly independent set.\\
Let $X\in 1_Ve_1$. Then, there exists $1_X\in 1_V$ such that $X=1_Xe_1$. Since $1_X\in 1_V$, therefore 
\begin{eqnarray*}
1_X=\sum_{i=1}^{k} \alpha_i (\beta_{i1},\beta_{i2},\ldots,\beta_{in}) 
&\Rightarrow& X= \left\{\sum_{i=1}^{k} \alpha_i (\beta_{i1},\beta_{i2},\ldots,\beta_{in})\right\}e_1
.\end{eqnarray*}
Using Remark \ref{th49}, it follows that
\begin{eqnarray*}
X&=&\left\{\sum_{i=1}^{k} (\alpha_i\beta_{i1},\alpha_i\beta_{i2},\ldots,\alpha_i \beta_{in})\right\}e_1\\
\ &=&\sum_{i=1}^{k} (\alpha_i\beta_{i1},\alpha_i\beta_{i2},\ldots,\alpha_i \beta_{in})e_1\\
&=&\sum_{i=1}^{k} ((\alpha_i\beta_{i1})e_1,(\alpha_i\beta_{i2})e_1,\ldots,(\alpha_i \beta_{in})e_1)\\
&=&\sum_{i=1}^{k} (\alpha_i(\beta_{i1}e_1),\alpha_i(\beta_{i2}e_1),\ldots,\alpha_i (\beta_{in}e_1))\\
&=&\sum_{i=1}^{k} \alpha_i(\beta_{i1}e_1,\beta_{i2}e_1,\ldots,\beta_{in}e_1)\\
&=&\sum_{i=1}^{k} \alpha_i\{(\beta_{i1},\beta_{i2},\ldots,\beta_{in})e_1\}
.\end{eqnarray*}
This implies that
\begin{eqnarray*}
&&X\in \langle S' \rangle\\
&\Rightarrow& 1_V e_1 \subseteq \langle S' \rangle \\
&\Rightarrow&  1_V e_1 = \langle S' \rangle    
.\end{eqnarray*}
Therefore, $S'$ is a basis for $1_V e_1$. This implies that $\dim(1_V e_1)=k$.\\ 
The proof of theorem is completed.
\end{proof}
Dually we may prove the following result.
\begin{theorem}\label{th46}
Let $2_V$ be a subspace of $\C_1^n(\C_1)$  and $\dim(2_V)=k$. Then $\dim(2_V e_2)=k$.        
\end{theorem}
The next result follows from the Corollary 2.46, Theorem 2.47 and Theorem 2.48.

\begin{theorem}\label{th47}
Let $1_V$ and $2_V$ be two sub-spaces of $\C_1^n(\C_1
)$, $\dim(1_V)=k_1$ and  $\dim(2_V)=k_2$. Then $\dim(1_V e_1 + 2_V e_2) = k_1 + k_2$.     
\end{theorem}
The next result follows from the Theorem \ref{th47}.
\begin{proposition}\label{th48}
If $A=1_A e_1 + 2_A e_2 \in \C_2^{n\times m}$, then \\
$\dim($idempotent column space of $A)$= $\dim($idempotent row space of $A)$, i.e. $\rho_{ir}(A)$ = $\rho_{ic}(A)$
.\end{proposition}
In the view of the results 2.36 of section 2. The following problem still remain open for the research.\\
{\bf Problem 1.}
If $A\in \C_2^{n\times n}$ is a non-singular matrix such that $\xi_{ij}\notin O_2$ for all $i, j$, then for $n\le 3$ we have shown that $\rho(A) = n$. Is this true for all n?

\bibliographystyle{amsplain}
\bibliography{references}

\providecommand{\bysame}{\leavevmode\hbox to3em{\hrulefill}\thinspace}
\providecommand{\MR}{\relax\ifhmode\unskip\space\fi MR }
\providecommand{\MRhref}[2]{%
  \href{http://www.ams.org/mathscinet-getitem?mr=#1}{#2}
}
\providecommand{\href}[2]{#2}
\begin{thebibliography}{10}

\bibitem{alpay2014basics}
D.~Alpay, M.~E. Luna-Elizarrar{\'a}s, M.~Shapiro, and D.~C Struppa, \emph{Basics of functional analysis with bicomplex scalars, and bicomplex schur analysis}, Springer Science \& Business Media, 2014.

\bibitem{anjali2023matrix}
Anjali, Fahed Zulfeqarr, Akhil Prakash, and Prabhat Kumar, \emph{Matrix representations of linear transformations on bicomplex space}, Asian-European Journal of Mathematics \textbf{17} (2024), no.~11, 2450079.

\bibitem{futagawa1928}
M.~Futagawa, \emph{On the theory of functions of a quaternary variable}, Tohoku Mathematical Journal, First Series \textbf{29} (1928), 175--222.

\bibitem{futagawa1932}
\bysame, \emph{On the theory of functions of a quaternary variable (part ii)}, Tohoku Mathematical Journal, First Series \textbf{35} (1932), 69--120.

\bibitem{gervais2011finite}
R.~Gervais~Lavoie, L.~Marchildon, and D.~Rochon, \emph{Finite-dimensional bicomplex hilbert spaces}, Advances in applied Clifford algebras \textbf{21} (2011), no.~3, 561--581.

\bibitem{luna2015bicomplex}
M~E. Luna-Elizarrar{\'a}s, M.~Shapiro, D.~C Struppa, and A.~Vajiac, \emph{Bicomplex holomorphic functions: the algebra, geometry and analysis of bicomplex numbers}, Birkh{\"a}user, 2015.

\bibitem{price2018introduction}
G.~B. Price, \emph{An introduction to multicomplex spaces and functions}, CRC Press, 2018.

\bibitem{riley1953}
J.~D. Riley, \emph{Contributions to the theory of functions of a bicomplex variable}, Tohoku Mathematical Journal, Second Series \textbf{5} (1953), no.~2, 132--165.

\bibitem{rochon2004}
Dominic Rochon and Michael Shapiro, \emph{On algebraic properties of bicomplex and hyperbolic numbers}, Anal. Univ. Oradea, fasc. math \textbf{11} (2004), no.~71, 110.

\bibitem{segre1892}
C.~Segre, \emph{Le rappresentazioni reali delle forme complesse e gli enti iperalgebrici}, Mathematische Annalen \textbf{40} (1892), no.~3, 413--467.

\bibitem{srivastava2008}
R.~K. Srivastava, \emph{Certain topological aspects of bicomplex space}, Bull. Pure \& Appl. Math \textbf{2} (2008), no.~2, 222--234.

\end{thebibliography}
\end{document}